\documentclass[12pt,twoside]{amsart}
\usepackage{amsmath, amsthm, amscd, amsfonts, amssymb, graphicx}
\usepackage{enumerate}
\usepackage[colorlinks=true,
linkcolor=blue,
urlcolor=cyan,
citecolor=red]{hyperref}
\usepackage{mathrsfs}
\newtheorem{theorem}{Theorem}

\newtheorem{lemma}{Lemma}[section]
\newtheorem{remark}{Remark}[section]

\newtheorem{corollary}{Corollary}[section]

\newtheorem{proposition}{Proposition}[section]
\numberwithin{equation}{section}

\begin{document}
\title{Exponential inequalities for positive linear mappings}
\author{M. Sababheh, H. R. Moradi and S.Furuichi}
 \maketitle

%------------------------------------------------------------------------------------%
\pagestyle{myheadings}
\markboth{\centerline {Exponential inequalities for positive linear mappings}}
{\centerline {}}
\bigskip
\bigskip
%------------------------------------------------------------------------------------%
%------------------------------------------------------------------------------------

\begin{abstract}
In this article, we present exponential-type inequalities for positive linear mappings and Hilbert space operators, by means of convexity and the Mond-Pe\v cari\'c method. The obtained results refine and generalize some  known results. As an application, we present extensions for operator-like geometric and harmonic means. 
\end{abstract}

\section{Introduction}
Let $\mathcal{B}(\mathcal{H})$ be the $C^*$ algebra of bounded linear operators on a complex Hilbert space $\mathcal{H}.$ If $A\in\mathcal{B}(\mathcal{H})$ is  positive, we write   $A\ge 0.$ Further, we use the notation $\mathcal{B}^{+}(\mathcal{H})$ for the cone of all positive operators in $\mathcal{B}(\mathcal{H}).$ For two self-ajoint operators $A,B\in\mathcal{B}(\mathcal{H})$, we  write  $A\le B$ if $B-A\geq 0$. For a real-valued function $f$ of a real variable and a self-adjoint operator $A\in \mathcal{B}\left( \mathcal{H} \right)$, the value $f\left( A \right)$ is understood by means of the functional calculus.\\

Let $J$ be a real interval of any type. A continuous function $f:J\to \mathbb{R}$ is said to be operator convex if $f\left( \left( 1-v \right)A+vB \right)\le \left( 1-v \right)f\left( A \right)+vf\left( B \right)$ holds for each $v\in \left[ 0,1 \right]$ and every pair of self-adjoint operators $A,B\in \mathcal{B}\left( \mathcal{H} \right)$, with spectra in $J$.
The notation $\mathcal{B}_{[m,M]}^{+}(\mathcal{H})$ will be used in the sequel to denote the class of all positive operators $A\in\mathcal{B}(\mathcal{H})$ satisfying $mI\leq A\leq MI$, for some positive scalars $m$ and $M$.
A linear map $\varphi:\mathcal{B}(\mathcal{H})\to \mathcal{B}(\mathcal{K})$ is said to be positive if $\varphi(A)\geq 0$ whenever $A\geq 0$. If, in addition, $\varphi(I)=I,$ it is said to be normalized.

If $f:J\to\mathbb{R}$ is operator convex, then for any normalized positive linear map $\varphi$, we have \cite{choi,davis}
\begin{equation}\label{oper_conc_intro}
f\left( \varphi \left( A \right) \right)\le \varphi \left( f\left( A \right) \right),
\end{equation}
while we have the reversed inequality if  $f$ is operator concave, for any self-adjoint operator $A$ with spectrum in $J$.

The inequality \eqref{oper_conc_intro} is not true if $f$ is a convex function (rather than operator convex). However, in the interesting paper \cite{micic_pecaric_2000}, various complementary inequalities have been presented for convex and concave functions. For example, it is shown that for any positive number $\alpha$, one can find a constant $\beta$ such that
\begin{equation}\label{pecaric_intro}
\varphi(f(A))\leq \alpha f(\varphi(A))+\beta,
\end{equation}
for the twice differentiable convex function $f:[m,M]\to\mathbb{R}$ and any self-adjoint operator $A$ on $\mathcal{H}$ with spectrum in $[m,M].$\\

Earlier,  it has been shown that for any continuous real valued function $f$, one can find positive constants $\alpha$ and $\beta$ such that  \cite{li}
\begin{equation}\label{li_intro}
\alpha\;\varphi(f(A))\leq f(\varphi(A))\leq \beta\;\varphi(f(A)).
\end{equation}

In Propsoition \ref{2nd_prop_pos_map}, we present a special case  \eqref{pecaric_intro} for a particular choice of $\alpha$, however we present a simple proof for completeness. Then as an application, we present several improvements  and extensions of \eqref{pecaric_intro} and \eqref{li_intro}  for a log-convex function $f$.
 
In the sequel, we adopt the following notations. For a given function $f:[m,M]\to\mathbb{R}$, define
\begin{equation}\label{def_L}
L[m,M,f](t)=a[m,M,f]\;t+b[m,M,f],
\end{equation}
where
\[a[m,M,f]=\frac{f\left( M \right)-f\left( m \right)}{M-m}\quad\text{ and }\quad b[m,M,f]=\frac{Mf\left( m \right)-mf\left( M \right)}{M-m}.\]
If no confusion arises, we will simply write $a[m,M,f]=a_f$ and $ b[m,M,f]=b_f.$

Also, for $t_0\in(m,M),$ define
\begin{equation}\label{def_L'}
L'[t_0,f](t)=f(t_0)+f'(t_0)(t-t_0),
\end{equation}
provided that $f'(t_0)$ exists.

It is clear that for a convex function $f:[m,M]\to \mathbb{R}$, one has 
\begin{equation}\label{1st_ineq_for_conv}
L'[t_0,f](t)\leq f(t)\leq L[m,M,f](t),
\end{equation}
while the inequalities are reversed for a concave function $f$.

\begin{remark}\label{remark_t_0}
Notice that if $f$ is convex on an interval $J$ containing $[m,M]$, then \eqref{1st_ineq_for_conv} is still valid for any $t_0\in J$. That is, $t_0$ does not need to be in $(m,M).$
\end{remark}

Now, if $f:[m,M] \to \mathbb{R}^+$ is log-convex, we have the inequality $\log f(t)\leq L[ m,  M, \log f](t),$ which simply reads as follows 
\begin{equation}\label{1st_ineq_log_conv}
f(t)\leq \left(f^{t-m}(M)f^{M-t}(m)\right)^{\frac{1}{M-m}}\leq L[m,M,f](t), \,\, m\leq t\leq M,
\end{equation}
where the second inequality is due to the arithmetic-geometric inequality.
We refer the reader to \cite{hamid_lama} for some detailed discussion of \eqref{1st_ineq_log_conv}.

Another useful observations about log-convex functions is the following. If $f$ is log-convex on $[m,M]$ and if $t_0\in(m,M)$ is such that $f(t_0)\not=0,$ \eqref{1st_ineq_for_conv} implies
$$L'[t_0,g](t)\leq g,\;{\text{where}}\;g=\log f.$$
Simplifying this inequality implies the following.
\begin{lemma}\label{lemma_log_conv}
Let $f:[m,M]\to\mathbb{R}^+$ be log-convex. If $f$ is differentiable at $t_0\in(m,M)$, then
$$f(t)\geq f(t_0)\exp\left[\frac{f'(t_0)}{f(t_0)}(t-t_0)\right],  m\leq t\leq M.$$
\end{lemma}

In this article, we present several inequalities for log-convex functions based on the Mond-Pe\v cari\'c method. In particular, we present inequalities that can be viewed as exponential inequalities for log-convex functions. More precisely, we present inequalities among the quantities 
$$\varphi(f(A)), f(\varphi(A)),  \varphi \left( {{\left( {{f}^{A-m}}\left( M \right){{f}^{M-A}}\left( m \right) \right)}^{\frac{1}{M-m}}} \right)$$ and
$$ {{\left( {{f}^{\varphi \left( A \right)-m}}\left( M \right){{f}^{M-\varphi \left( A \right)}}\left( m \right) \right)}^{\frac{1}{M-m}}}.$$

Another interest in this paper is to present inequalities for operator-like means when filtered through normalized positive linear maps. That is, it is known that for an operator mean $\sigma$, one has \cite{mond}
$$\varphi(A\sigma B)\leq \varphi(A)\sigma \varphi(B), A,B\in\mathcal{B}^{+}(\mathcal{H}).$$
In particular, we show complementary inequalities for the geometric $\sharp_t$ and harmonic $!_t$ operator-like means, when $t<0.$ Of course, when $t<0$, these are not operator means. Our results can be considered as extensions of \cite[Theorem 2.2]{fujii}.

\section{Main Results}
Now we proceed to the main results, starting with a complementary result of \cite[Corollary 2.5]{2} and \cite[Proposition 2.1]{hamid_lama}.

\begin{proposition}\label{prop_1st_hamid}
Let $f:[m,M]\to\mathbb{R}^+$ be log-convex, $A\in\mathcal{B}_{[m,M]}^{+}(\mathcal{H})$ and $\varphi$ be a normalized positive linear map. Then
\[\begin{aligned}
 \frac{1}{\mu \left( m,M,f \right)}\varphi \left( f\left( A \right) \right)&\le \frac{1}{\mu \left( m,M,f \right)}\varphi \left( {{\left( {{f}^{A-m}}\left( M \right){{f}^{M-A}}\left( m \right) \right)}^{\frac{1}{M-m}}} \right) \\ 
& \le f\left( \varphi \left( A \right) \right) \\ 
& \le {{\left( {{f}^{\varphi \left( A \right)-m}}\left( M \right){{f}^{M-\varphi \left( A \right)}}\left( m \right) \right)}^{\frac{1}{M-m}}} \\ 
& \le \mu \left( m,M,f \right)\varphi \left( f\left( A \right) \right)  
\end{aligned}\]
where
\[\mu \left( m,M,f \right)\equiv \max \left\{ \frac{1}{f\left( t \right)}\left( \frac{M-t}{M-m}f\left( m \right)+\frac{t-m}{M-m}f\left( M \right) \right):\text{ }m\le t\le M \right\}.\]
\end{proposition}
\begin{proof}
The first and the second inequalities follow from \cite[Proposition 2.1]{hamid_lama} and the fact that $\mu \left( m,M,f \right)>0$. So we have to prove the other inequalities. Applying a standard functional calculus argument  for the operator $\varphi \left( A \right)$ in \eqref{1st_ineq_log_conv}, we get
\begin{equation}\label{needed-referee}
f\left( \varphi \left( A \right) \right)\le {{\left( {{f}^{\varphi \left( A \right)-m}}\left( M \right){{f}^{M-\varphi \left( A \right)}}\left( m \right) \right)}^{\frac{1}{M-m}}}\le L\left[ m,M,f \right]\left( \varphi \left( A \right) \right).
\end{equation}

Following \cite{2}, we have for $\alpha>0$,
\begin{align*}
L\left[ m,M,f \right]\left( \varphi \left( A \right) \right)-\alpha\;\varphi(f(A))&=a_f\;\varphi(A)+b_f-\alpha\;\varphi(f(A))\\
&\leq \beta,
\end{align*}
where $\beta=\max\limits_{t\in[m,M]}\{a_f\;t+b_f-\alpha\;f(t)\}.$ That is,
$$L\left[ m,M,f \right]\left( \varphi \left( A \right) \right)\leq \alpha\;\varphi(f(A))+\beta.$$
By setting $\beta=0,$ we obtain $\alpha=\max\limits_{t\in[m,M]}\left\{\frac{a_ft+b_f}{f(t)}\right\}:=\mu(m,M,f).$ With this choice of $\alpha$, we have $L\left[ m,M,f \right]\left( \varphi \left( A \right) \right)\leq \mu(m,M,f)\;\varphi(f(A)),$ which, together with \eqref{needed-referee}, complete the proof.
\end{proof}

Notice that Proposition \ref{prop_1st_hamid} can be regarded as an operator extenstion of \cite[Theorem 2.5]{dragomir} and a refinement of \cite[Corollary 2.5]{2}.
\begin{corollary}\label{cor-2-1}
	Let $A\in\mathcal{B}_{[m,M]}^{+}(\mathcal{H})$ and $\varphi$ be a normalized positive linear map. Then, for $t<0$, 
\[\begin{aligned}
 \frac{1}{K\left( m,M,t \right)}\varphi \left( {{A}^{t}} \right)&\le \frac{1}{K\left( m,M,t \right)}\varphi \left( {{\left( {{M}^{A-m}}{{m}^{M-A}} \right)}^{\frac{t}{M-m}}} \right) \\ 
& \le \varphi {{\left( A \right)}^{t}} \\ 
& \le {{\left( {{M}^{\varphi \left( A \right)-m}}{{m}^{M-\varphi \left( A \right)}} \right)}^{\frac{t}{M-m}}} \\ 
& \le K\left( m,M,t \right)\varphi \left( {{A}^{t}} \right)  
\end{aligned}\]
where the generalized Kantrovich constant is defined by
$$K(m,M,t)=\frac{(mM^{t}-Mm^{t})}{(t-1)(M-m)}\left(\frac{t-1}{t}\frac{M^t-m^t}{mM^t-Mm^t}\right)^{t}.$$
\end{corollary}
\begin{proof}
The result follows immediately from Proposition \ref{prop_1st_hamid}, be letting $f(x)=x^t.$
\end{proof}

\begin{remark}
Corollary \ref{cor-2-1} presents a refinement of the corresponding result in \cite[Lemma 2]{01}. 
\end{remark}

As another application of Proposition \ref{prop_1st_hamid}, we have the following bounds for  operator means. To simplify our statement, we will adopt the following notations. For a given function $f:[m,M]\to[0,\infty)$ and two positive operators $A$ and $B$ satisfying $m\;A\leq B\leq M\;A$, we write
$$A\sigma_fB:=A^{\frac{1}{2}}f\left(A^{-\frac{1}{2}}BA^{-\frac{1}{2}}\right)A^{-\frac{1}{2}}\;{\text{and}}\;A\delta B=A^{-\frac{1}{2}}BA^{-\frac{1}{2}}.$$

\begin{corollary}\label{cor_means}
Let $A,B\in\mathcal{B}^{+}(\mathcal{H})$ be such that $m\;A\leq B\leq M\;A$ for some positive scalars $m,M.$ Then, for any linear map $\varphi$ (not necessarily normalized) and any log-convex function $f:[m,M]\to\mathbb{R}^+$,

\begin{align*}
 \frac{1}{\mu \left( m,M,f \right)}\varphi \left(A\sigma_fB\right)&\le \frac{1}{\mu \left( m,M,f \right)}\varphi \left(A^{\frac{1}{2}} {{\left( {{f}^{A\delta B-m}}\left( M \right){{f}^{M-A\delta B}}\left( m \right) \right)}^{\frac{1}{M-m}}}A^{\frac{1}{2}} \right) \\ 
& \le \varphi(A)\sigma_f\varphi(B) \\ 
& \le \varphi(A)^{\frac{1}{2}}{{\left( {{f}^{\varphi(A)\delta \varphi(B)-m}}\left( M \right){{f}^{M-\varphi(A)\delta \varphi(B)}}\left( m \right) \right)}^{\frac{1}{M-m}}}\varphi(A)^{\frac{1}{2}} \\ 
& \le \mu \left( m,M,f \right)\varphi(A\sigma_f B).
\end{align*}
\end{corollary}
\begin{proof}
From the assumption $m\;A\leq B\leq M\;A$, we have $m\leq A\delta B:=A^{-\frac{1}{2}}BA^{-\frac{1}{2}}\leq M.$ Therefore, if $f$ is log-convex on $[m,M],$  Proposition \ref{prop_1st_hamid} implies
\[\begin{aligned}
 \frac{1}{\mu \left( m,M,f \right)}\psi \left( f\left( A\delta B \right) \right)&\le \frac{1}{\mu \left( m,M,f \right)}\psi \left( {{\left( {{f}^{A\delta B-m}}\left( M \right){{f}^{M-A\delta B}}\left( m \right) \right)}^{\frac{1}{M-m}}} \right) \\ 
& \le f\left( \psi \left( A\delta B \right) \right) \\ 
& \le {{\left( {{f}^{\psi \left( A\delta B \right)-m}}\left( M \right){{f}^{M-\psi \left( A\delta B \right)}}\left( m \right) \right)}^{\frac{1}{M-m}}} \\ 
& \le \mu \left( m,M,f \right)\psi \left( f\left( A\delta B \right) \right),
\end{aligned}\]
for any normalized positive linear map $\psi.$ In particular, for the given $\varphi,$  define $$\psi(X)=\varphi(A)^{-\frac{1}{2}}\varphi\left(A^{\frac{1}{2}}XA^{\frac{1}{2}}\right)\varphi(A)^{-\frac{1}{2}}.$$ Then,  $\psi$ is a normalized linear mapping and the above inequalities imply, upon conjugating with $\varphi(A)^{\frac{1}{2}},$ the desired inequalities.
\end{proof}
In particular, Corollary \ref{cor_means} can be utilized to obtain versions for the geometric and harmonic operator means, as follows.

\begin{corollary}\label{cor_geometric}
Let $A,B\in\mathcal{B}^{+}(\mathcal{H})$ be such that $m\;A\leq B\leq M\;A$ for some positive scalars $m,M.$ Then, for any linear map $\varphi$ (not necessarily normalized) and for $t<0,$
\begin{align*}
\frac{1}{K(m,M,t)}\varphi(A\sharp_t B)&\leq \frac{1}{K(m,M,t)}\varphi(A\sigma_gB)   \\
&\leq \varphi(A)\sharp_t\varphi(B)\\
&\leq \varphi(A)\sigma_g\varphi(B)\\
&\leq K(m,M,t)\varphi(A\sharp_tB),
\end{align*}
where
$g(x)=\left(M^{x-m}m^{M-x}\right)^{\frac{t}{M-m}}$ and $K(m,M,t)$ is as in Corollary \ref{cor-2-1}.
\end{corollary}
\begin{proof}
Noting that the function $f(x)=x^{t}$ is log-convex on $[m,M]$ for $t<0$, the result follows by direct application of Corollary \ref{cor_means}.
\end{proof}
\begin{remark}
	Recently in \cite[Theorem 2.2]{fujii} the authors proved that if $A,B$ are two positive operators, then
\[\varphi \left( A \right){{\sharp}_{t}}\varphi \left( B \right)\le \varphi \left( A{{\sharp}_{t}}B \right)\text{ }t\in \left[ -1,0 \right).\]
Therefore, Corollary \ref{cor_geometric} can be regarded as an extension and a reverse for the above inequality, under the assumption $mA \le B \le MA$ with $M \ge m >0$.
\end{remark}
\begin{corollary}
Let $A,B\in\mathcal{B}^{+}(\mathcal{H})$ be such that $m\;A\leq B\leq M\;A$ for some positive scalars $m,M\geq 1.$ Then, for any linear map $\varphi$ (not necessarily normalized) and for $t<0,$
\begin{align*}
\frac{1}{H(m,M,t)}\varphi(A!_t B)&\leq \frac{1}{H(m,M,t)}\varphi(A\sigma_gB)   \\
&\leq \varphi(A)!_t\varphi(B)\\
&\leq \varphi(A)\sigma_g\varphi(B)\\
&\leq H(m,M,t)\;\varphi(A!_tB),
\end{align*}
where $g(x)=\left(\left(1!_{t}M\right)^{x-m}\left(1!_{t}m\right)^{M-x}\right)^{\frac{1}{M-m}}$ and
$$H(m,M,t)=\left[(1-t)^2+\frac{t}{mM}\left(2(1-t)\sqrt{m M}+t\right)\right](1!_tm)(1!_tM).$$
\end{corollary}
\begin{proof}
Noting that the function $f(x)=(1-t+t\;x^{-1})^{-1}$ is log-convex on $[m,M]$ for $t<0$, provided that $m\geq 1,$ the result follows by direct application of Corollary \ref{cor_means}.
\end{proof}
We should remark that the mapping $t\mapsto H(m,M,t)$ is a decreasing function for $t<0.$ In particular, $$H(m,M,0)=1\leq H(m,M,t)\leq \frac{(\sqrt{mM}-1)^2}{(m-1)(M-1)}=\lim_{t\to-\infty}H(m,M,t), \forall \;t<0.$$

Further, utilizing \eqref{1st_ineq_log_conv}, we obtain the following. In this result and later in the paper, we adopt the notations:
$$\alpha(f,t_0)=\frac{a_f}{f'(t_0)}\;{\text{and}}\;\beta(f,t_0)=a_ft_0+b_f-\frac{a_f\;f(t_0)}{f'(t_0)}.$$

The following Proposition gives a simplified special case of \cite[Theorem 2.1]{micic_pecaric_2000}.
\begin{proposition}\label{2nd_prop_pos_map}
Let $f:[m,M]\to\mathbb{R}^+$ be convex, $A\in\mathcal{B}_{[m,M]}^{+}(\mathcal{H})$ and $\varphi$ be a normalized positive linear map. If $f$ is either increasing or decreasing on $[m,M],$ then for any $t_0\in (m,M),$
\begin{equation}\label{2ndd_ineq_pos_map}
\varphi(f(A))\leq \alpha(f,t_0)f(\varphi(A))+\beta(f,t_0),
\end{equation}
and
\begin{equation}\label{22nd_ineq_pos_map}
f(\varphi(A))\leq \alpha(f,t_0)\varphi(f(A))+\beta(f,t_0),
\end{equation}
provided that $f'(t_0)$ exists and $f'(t_0)\not=0.$ Further, both inequalities are reversed if $f$ is concave.
\end{proposition}
\begin{proof}
We give the proof for the reader's convenience.
Notice first that $f$ being either increasing or decreasing assures that $a_ff'(t_0)>0.$
Using a standard functional calculus in \eqref{1st_ineq_for_conv} with $t=A$ and applying $\varphi$ to both sides imply
\begin{equation}\label{needed_ineq_1}
f(t_0)+f'(t_0)(\varphi(A)-t_0)\leq \varphi(f(A))\leq a_f\varphi(A)+b_f.
\end{equation} 
On the other hand, applying the functional calculus argument with $t=\varphi(A)$ implies
\begin{equation}\label{needed_ineq_2}
f(t_0)+f'(t_0)(\varphi(A)-t_0)\leq f(\varphi(A))\leq a_f\varphi(A)+b_f.
\end{equation}
Noting that $a_f$ and $f'(t_0)$ have the same sign, both desired inequalities follow from \eqref{needed_ineq_1} and \eqref{needed_ineq_2}.

Now if $f$ was concave, replacing $f$ with $-f$ and noting linearity of $\varphi$ imply the desired inequalities for a concave function.
\end{proof}

As an application, we present the following result, which has been shown in  \cite[Corollary 2.8]{micic_pecaric_2000}. 
\begin{corollary}\label{cor_inverse}
Let $A\in\mathcal{B}_{[m,M]}^{+}(\mathcal{H})$. Then, for a normalized positive linear mapping $\varphi,$
\begin{equation}\label{additive_reverse_inverse}
\varphi(A^{-1})\leq (\varphi(A))^{-1}+\left(\frac{1}{\sqrt{m}}-\frac{1}{\sqrt{M}}\right)^{2}
\end{equation}
and
\begin{equation}\label{multi_reverse_inverse}
\varphi(A^{-1})\leq \frac{(M+m)^2}{4mM}(\varphi(A))^{-1}.
\end{equation}
\end{corollary}
\begin{proof}
Let $f(t)=t^{-1}.$ Then $f$ is convex and monotone on $[m,M].$ Letting $t_0=\sqrt{mM}\in(m,M),$ direct calculations show that $\alpha(f,t_0)=1, \beta(f,t_0)=\left(\frac{1}{\sqrt{m}}-\frac{1}{\sqrt{M}}\right)^{2}.$ Then inequality \eqref{2ndd_ineq_pos_map} implies the first inequality. The second inequality follows simlarly by letting $t_0=\frac{m+M}{2}.$
\end{proof}

Manipulating Proposition \ref{2nd_prop_pos_map} implies several extensions for log-convex functions, as we shall see next.

 We will adopt the following constants in Theorem \ref{theorem_1}.\\ 
 $a_h=a[m,M,h], b_h=b[m,M,h], \alpha=\alpha(h,t_0), \beta=\beta(h,t_0)$ for $h(t)=\left(f^{t-m}(M)f^{M-t}(m)\right)^{\frac{1}{M-m}}$ and $a_{h_1}=a[f^{M-m}(m),f^{M-m}(M),h_1], b_{h_1}=b[f^{M-m}(m),f^{M-m}(M),h_1], \alpha_1=\alpha(h_1,t_1)$ and $\beta_1=\beta(h_1,t_1)$ for $h_1(t)=t^{\frac{1}{M-m}}.$

 The first two inequalities of the next result should be compared with Proposition \ref{prop_1st_hamid}; where a reverse-type is presented now.
\begin{theorem}\label{theorem_1}
Let $f:[m,M]\to\mathbb{R}^+$ be log-convex, $A\in\mathcal{B}_{[m,M]}^{+}(\mathcal{H})$ and $\varphi$ be a normalized positive linear map. Then for any $t_0,t_1\in (m,M),$
\begin{align*}
f(\varphi(A))&\leq \left(f^{\varphi(A)-m}(M)f^{M-\varphi(A)}(m)\right)^{\frac{1}{M-m}}\\
&\leq \alpha\;\varphi\left(f^{\frac{A-m}{M-m}}(M)f^{\frac{M-A}{M-m}}(m)\right)+\beta\\
&\leq \left\{\begin{array}{cc}\alpha\;\left(\varphi\left(f^{A-m}(M)f^{M-A}(m)\right)\right)^{\frac{1}{M-m}}+\beta,&M-m\geq 1 \\
\alpha\;\alpha_1\left(\varphi\left(f^{A-m}(M)f^{M-A}(m)\right)\right)^{\frac{1}{M-m}}+\alpha\beta_1+\beta,&M-m<1\end{array}\right..
\end{align*}
\end{theorem}
\begin{proof}
For $h(t)=\left(f^{t-m}(M)f^{M-t}(m)\right)^{\frac{1}{M-m}}$, we clearly see that $h$ is convex  and monotone on $[m,M]$. Notice that
\begin{align*}
f(\varphi(A))&\leq \left(f^{\varphi(A)-m}(M)f^{M-\varphi(A)}(m)\right)^{\frac{1}{M-m}} \quad({\text{by\;the\;third\;inequality\;of\;Proposition}}\;\ref{prop_1st_hamid})\\
&=h(\varphi(A))\\
&\leq \alpha(h,t_0)\;\varphi(h(A))+\beta(h,t_0)\quad({\text{by}}\;\eqref{22nd_ineq_pos_map})\\
&=\alpha(h,t_0)\;\varphi\left(f^{\frac{A-m}{M-m}}(M)f^{\frac{M-A}{M-m}}(m)\right)+\beta(h,t_0).
\end{align*}
This proves the first two inequalities. Now, for the third inequality, assume that $M-m\geq 1$ and let $h_1(t)=t^{\frac{1}{M-m}}.$ Then the second inequality can be viewed as
\begin{equation}\label{needed_2nd_prop}
f(\varphi(A))\leq \alpha\; \varphi\left(h_1\left(f^{A-m}(M)f^{M-A}(m)\right)\right)+\beta.
\end{equation}
Since $M-m\geq 1,$ it follows that $h_1$ is operator concave. Therefore, noting \eqref{needed_2nd_prop} and \eqref{oper_conc_intro}, we have
\begin{align*}
f(\varphi(A))&\leq \alpha\; \varphi\left(h_1\left(f^{A-m}(M)f^{M-A}(m)\right)\right)+\beta\\
&\leq \alpha\;h_1\left(\varphi\left(f^{A-m}(M)f^{M-A}(m)\right)\right)+\beta,
\end{align*}
which is the desired inequality in the case $M-m\geq 1.$\\
Now, if $ M-m< 1,$ the function $h_1$ is  convex and monotone. Therefore, taking in account  \eqref{needed_2nd_prop} and \eqref{2ndd_ineq_pos_map}, we obtain
\begin{align*}
f(\varphi(A))&\leq \alpha\; \varphi\left(h_1\left(f^{A-m}(M)f^{M-A}(m)\right)\right)+\beta\\
&\leq \alpha \left(\alpha_1\;h_1\left(\varphi\left(f^{A-m}(M)f^{M-A}(m)\right)\right)+\beta_1\right)+\beta\\
&=\alpha\;\alpha_1\left(\varphi\left(f^{A-m}(M)f^{M-A}(m)\right)\right)^{\frac{1}{M-m}}+\alpha\beta_1+\beta,
\end{align*}
which completes the proof.
\end{proof}

For the same parameters as Theorem \ref{theorem_1}, we have the following comparison too, in which the first two inequalities have been shown in Theorem \ref{theorem_1}.
\begin{corollary}\label{cor_thm_1}
Let $f:[m,M]\to\mathbb{R}^+$ be log-convex, $A\in\mathcal{B}_{[m,M]}^{+}(\mathcal{H})$ and $\varphi$ be a normalized positive linear map. Then 
\begin{align*}
f(\varphi(A))&\leq \left(f^{\varphi(A)-m}(M)f^{M-\varphi(A)}(m)\right)^{\frac{1}{M-m}}\\
&\leq \alpha\;\varphi\left(f^{\frac{A-m}{M-m}}(M)f^{\frac{M-A}{M-m}}(m)\right)+\beta\\
&\leq \alpha\;\mu(m,M,f) \;\varphi(f(A))+\beta.
\end{align*}
\end{corollary}
\begin{proof}
We prove the last inequality. Letting $\psi(X)=X$ be a normalized positive linear map and noting that $\varphi$ is order preserving, the fourth inequality of proposition \ref{prop_1st_hamid} implies
\begin{align*}
\alpha\;\varphi\left(f^{\frac{A-m}{M-m}}(M)f^{\frac{M-A}{M-m}}(m)\right)+\beta&=\alpha\;\varphi\left(f^{\frac{\psi(A)-m}{M-m}}(M)f^{\frac{M-\psi(A)}{M-m}}(m)\right)+\beta\\
&\leq \alpha\;\mu(m,M,f)\;\varphi\left( \psi(f(A)) \right)+\beta\\
&=\alpha\;\mu(m,M,f)\;\varphi\left( f(A) \right)+\beta,
\end{align*}
which is the desired inequality.
\end{proof}

For the next result, the following constants will be used.\\
$\hat{a}_h=a[m,M,\hat{h}], \hat{b}_h=b[m,M,\hat{h}], \hat{\alpha}=\hat{\alpha}(\hat{h},t_0), \hat{\beta}=\hat{\beta}(\hat{h},t_0)$ for $\hat{h}(t)=f^{t-m}(M)f^{M-t}(m)$ and $a_{h_1}=a[f^{M-m}(m),f^{M-m}(M),h_1], b_{h_1}=b[f^{M-m}(m),f^{M-m}(M),h_1], \alpha_1=\alpha(h_1,t_1)$ and $\beta_1=\beta(h_1,t_1)$ for $h_1(t)=t^{\frac{1}{M-m}}.$
\begin{theorem}\label{theorem_2}
Let $f:[m,M]\to\mathbb{R}^+$ be log-convex, $t_0,t_1\in (m,M)$, $A\in\mathcal{B}_{[m,M]}^{+}(\mathcal{H})$ and $\varphi$ be a normalized positive linear map. If $M-m\geq 1,$
\begin{align*}
\varphi(f(A))&\leq \left[\varphi\left(f^{A-m}(M)f^{M-A}(m)\right)\right]^{\frac{1}{M-m}}\\
&\leq \left(\hat{\alpha}\;f^{\varphi(A)-m}(M)f^{M-\varphi(A)}(m)+\hat{\beta}\right)^{\frac{1}{M-m}},
\end{align*}

On the other hand, if $M-m< 1,$
\begin{align*}
\varphi(f(A))&\leq \alpha_1\left[\varphi\left(f^{M-A}(m)f^{A-m}(M)\right) \right]^{\frac{1}{M-m}}+\beta_1\\
&\leq \alpha_1\left[\hat{\alpha}\;f^{\varphi(A)-m}(M)f^{M-\varphi(A)}(m)+\hat{\beta} \right]^{\frac{1}{M-m}}+\beta_1.
\end{align*}
\end{theorem}
\begin{proof}
Letting $h_1(t)=t^{\frac{1}{M-m}}$ and $\hat{h}(t)=f^{t-m}(M)f^{M-t}(m),$ we have
\begin{align*}
\varphi(f(A))&\leq \varphi\left(h_1\left(f^{A-m}(M)f^{M-A}(m)\right)\right)\quad({\text{by\;the\;first\;inequality\;of\;Proposition}}\;\ref{prop_1st_hamid})\\
&\leq h_1\left(\varphi\left(f^{A-m}(M)f^{M-A}(m)\right)\right)\quad({\text{since}}\;h_1\;{\text{is\;operator\;concave}})\\
&= \left[\varphi\left(f^{A-m}(M)f^{M-A}(m)\right)\right]^{\frac{1}{M-m}}\\
&=\left[\varphi(\hat{h}(A))\right]^{\frac{1}{M-m}}\quad({\text{where}}\;\hat{h}(t)=f^{t-m}(M)f^{M-t}(m))\\
&\leq \left[\hat{\alpha}\;\hat{h}(\varphi(A))+\beta\right]^{\frac{1}{M-m}}\quad({\text{by}}\;\eqref{2ndd_ineq_pos_map})\\
&=\left(\hat{\alpha}\;f^{\varphi(A)-m}(M)f^{M-\varphi(A)}(m)+\hat{\beta}\right)^{\frac{1}{M-m}},
\end{align*}
which completes the proof for the case $M-m\geq 1.$

Now if $M-m<1$, we have
\begin{align*}
\varphi(f(A))&\leq \varphi\left(h_1\left(f^{A-m}(M)f^{M-A}(m)\right)\right)\quad \;({\text{by\;the\;first\;inequality\;of\;Proposition}}\;\ref{prop_1st_hamid})\\
&\leq \alpha_1\;h_1\left(\varphi\left(f^{A-m}(M)f^{M-A}(m)\right)\right)+\beta_1\quad ({\text{by}}\;\eqref{2ndd_ineq_pos_map})\\
&=\alpha_1\;\left[\varphi\left(f^{A-m}(M)f^{M-A}(m)\right)\right]^{\frac{1}{M-m}}+\beta_1\\
&=\alpha_1\left[\varphi(\hat{h}(A))\right]^{\frac{1}{M-m}}+\beta_1\\
&\leq \alpha_1\left[\hat{\alpha}\;\hat{h}(\varphi(A))+\hat{\beta}\right]^{\frac{1}{M-m}}+\beta_1\quad ({\text{by}}\;\eqref{2ndd_ineq_pos_map})\\
&=\alpha_1\left[\hat{\alpha}\;f^{\varphi(A)-m}(M)f^{M-\varphi(A)}(m)+\hat{\beta} \right]^{\frac{1}{M-m}}+\beta_1,
\end{align*}
which completes the proof.
\end{proof}

\begin{remark}
In both Theorems \ref{theorem_1} and \ref{theorem_2}, the constants $\alpha$ and $\alpha_1$ can be selected to be 1, as follows. Noting that the function $h$ in both theorems is continuous on $[m,M]$ and differentiable on $(m,M)$, the mean value theorem assures that $a_h=h'(t_0)$  for some $t_0\in (m,M).$ This implies $\alpha=1$, since we use the notation $\alpha \equiv \alpha(h,t_0) = \frac{a_h}{h'(t_0)}$. A similar argument applies for $h_1.$ These values of $t_0$ can be easily found.\\
Moreover, one can find $t_0$ so that $\beta(h,t_0)=0,$ providing a multiplicative version. Since this is a direct application, we leave the tedious computations to the interested reader. 
\end{remark}

Utilizing Lemma \ref{lemma_log_conv}, we obtain the following exponential inequality.
\begin{proposition}\label{e_prop_oper}
Let $f:[m,M]\to\mathbb{R}^+$ be log-convex, $t_0,t_1\in (m,M)$, $A\in\mathcal{B}_{[m,M]}^{+}(\mathcal{H})$ and $\varphi$ be a normalized positive linear map. Then
\begin{align*}
\varphi(f(A))&\geq \frac{f(t_0)}{\alpha}\exp\left[\frac{f'(t_0)}{f(t_0)}(\varphi(A)-t_0)\right]-\frac{\beta}{\alpha}f(t_0),
\end{align*}
and 
\begin{align*}
f(\varphi(A))&\geq \frac{f(t_0)}{\alpha}\varphi\left(\exp\left[\frac{f'(t_0)}{f(t_0)}(A-t_0)\right]\right)-\frac{\beta}{\alpha}f(t_0),
\end{align*}
where $\alpha=\alpha(k,t_1)$ and $\beta=\beta(k,t_1)$ for $k(t)=\exp\left[\frac{f'(t_0)}{f(t_0)}(t-t_0)\right].$
\end{proposition}
\begin{proof}
By Lemma \ref{lemma_log_conv}, we have
$$f(t)\geq f(t_0)\exp\left[\frac{f'(t_0)}{f(t_0)}(t-t_0)\right],  m\leq t\leq M.$$ A functional calculus argument applied to this inequality with $t=A$ implies
\begin{align*}
\varphi(f(A))&\geq f(t_0)\;\varphi\left(\exp\left[\frac{f'(t_0)}{f(t_0)}(A-t_0)\right]\right)\\
&=f(t_0)\;\varphi(k(A))\\
&\geq f(t_0)\frac{k(\varphi(A))-\beta}{\alpha}\quad({\text{by}}\;\eqref{22nd_ineq_pos_map})\\
&=\frac{f(t_0)}{\alpha}\exp\left[\frac{f'(t_0)}{f(t_0)}(\varphi(A)-t_0)\right]-\frac{\beta}{\alpha}f(t_0),
\end{align*}
which completes the proof of the first inequality. The second inequality follows similarly using \eqref{2ndd_ineq_pos_map}.
\end{proof}

\section{Further refinements}
The above results are all based on basic inequalities for convex functions. Therefore, refinements of convex functions inequalities can be used to obtain sharper bounds. We give here some examples. In \cite{mitroi}, the following simple inequality was shown for the convex function $f:[m,M]\to\mathbb{R}$,
\begin{equation}\label{mitroi_ineq}
f(t)+\frac{2\min\{t-m,M-t\}}{M-m}\left(\frac{f(m)+f(M)}{2}-f\left(\frac{m+M}{2}\right)\right)\leq L[m,M,f](t).
\end{equation}
This inequality can be used to obtain refinements of \eqref{2ndd_ineq_pos_map} and \eqref{22nd_ineq_pos_map} as follows.
First, we note that the function $t\mapsto t_{\min}:=\frac{2\min\{t-m,M-t\}}{M-m}\left(\frac{f(m)+f(M)}{2}-f\left(\frac{m+M}{2}\right)\right)$ is a continuous function. Further, noting that $$\min\{t-m,M-t\}=\frac{M-m+|M+m-2t|}{2},$$ one can apply a functional calculus argument on \eqref{mitroi_ineq}. With this convention, we will use the notation
$$A_{\min}:=\frac{1}{M-m}\left(\frac{f(m)+f(M)}{2}-f\left(\frac{m+M}{2}\right)\right)\left(M-m+|M+m-A|\right).$$
The following is a refinement of of Proposition \ref{2nd_prop_pos_map}. Since the proof is similar to that of Proposition \ref{2nd_prop_pos_map} utilizing \eqref{mitroi_ineq}, we do not include it here.
\begin{proposition}
Let $f:[m,M]\to\mathbb{R}^+$ be convex, $A\in\mathcal{B}_{[m,M]}^{+}(\mathcal{H})$ and $\varphi$ be a normalized positive linear map. If $f$ is either increasing or decreasing on $[m,M],$ then for any $t_0\in (m,M),$
\begin{equation}\label{2ndd_ineq_pos_map_ref}
\varphi(f(A))+\varphi(A_{\min})\leq \alpha(f,t_0)f(\varphi(A))+\beta(f,t_0),
\end{equation}
and
\begin{equation}\label{22nd_ineq_pos_map_ref}
f(\varphi(A))+(\varphi(A))_{\min}\leq \alpha(f,t_0)\varphi(f(A))+\beta(f,t_0),
\end{equation}
provided that $f'(t_0)$ exists and $f'(t_0)\not=0.$
\end{proposition}
Notice that applying this refinement to the convex function $f(t)=t^{-1}$ implies refinements of both inequalities in Corollary \ref{cor_inverse} as follows.
\begin{corollary}
Under the assumptions of Corollary  \ref{cor_inverse}, we have
\begin{equation}\label{additive_reverse_inverse_ref}
\varphi(A^{-1})+\varphi(A_{\min})\leq (\varphi(A))^{-1}+\left(\frac{1}{\sqrt{m}}-\frac{1}{\sqrt{M}}\right)^{2}
\end{equation}
and
\begin{equation}\label{multi_reverse_inverse_ref}
\varphi(A^{-1})+(\varphi(A))_{\min}\leq \frac{(M+m)^2}{4mM}(\varphi(A))^{-1}.
\end{equation}
\end{corollary}

\begin{remark}
The inequality \eqref{mitroi_ineq} has been studied extensively in the literature, where numerous refining terms have been found. We refer the reader to \cite{sab_mia} and \cite{sab_mjom}, where a comprehensive discussion has been made therein. These refinements then can be used to obtain further refining terms for Proposition \ref{2nd_prop_pos_map}.\\
Further, these refinements can be applied to log-convex functions too. This refining approach leads to refinements of most inequalities presented in this article; where convexity was the key idea. We leave the detailed computations to the interested reader.
\end{remark}

\section*{Data Availability}
All data generated or analysed during this study are included in this published article.
There is no experimental data in this article.

\section*{Acknowledgment}
The authors thank anonymous referees for giving valuable comments and suggestions to improve our manuscript. 
The work of the first author (M. S.) is supported by a sabbatical leave from Princess Sumaya University for technology. The author (S.F.) was partially supported by JSPS KAKENHI Grant Number 16K05257.

%-----------------------------------------------------------------------------
%-----------------------------------------------------------------------------
\vskip 0.8 true cm 

{\tiny (Mohammad Sababheh) Department of Basic Sciences, Princess Sumaya University for Technology, Amman, 	Jordan.

	\textit{E-mail address:} sababheh@yahoo.com, sababheh@psut.edu.jo 

\vskip 0.3 true cm 	
	
{\tiny (Hamid Reza Moradi) Department of Mathematics, Payame Noor Universtiy (PNU), P.O. BOX: 19395-4697, Tehran, Iran.}
	
	{\tiny \textit{E-mail address:} hrmoradi@mshdiau.ac.ir }
	
{\tiny \vskip 0.3 true cm }

{\tiny (Shigeru Furuichi) Department of Computer Science and System Analysis, College of Humanities and Sciences, Nihon University, 3-25-40, Sakurajyousui, Setagaya-ku, Tokyo 156-8550, Japan}
	
	{\tiny \textit{E-mail address:} furuichi@chs.nihon-u.ac.jp}
	
{\tiny \vskip 0.3 true cm }

\end{document}